\documentclass[12pt]{amsart}
\usepackage{graphicx}
\usepackage{amsmath,amssymb}
\newtheorem{theorem}{Theorem}[section]
\newtheorem{proposition}[theorem]{Proposition}
\newtheorem{lemma}[theorem]{Lemma}

\theoremstyle{remark}
\newtheorem{remark}[theorem]{Remark}

\numberwithin{equation}{section}

\begin{document}

\title[
On linear $n$-colorings for knots
]{
On linear $n$-colorings for knots
}

\author{Chuichiro Hayashi, Miwa Hayashi and Kanako Oshiro}


\thanks{The first author is partially supported
by Grant-in-Aid for Scientific Research (No. 22540101),
Ministry of Education, Science, Sports and Technology, Japan.}

\thanks{The third author is partially supported
by Grant-in-Aid for Research Activity Start-up (No. 90609091),
Japan Society for the Promotion of Science}

\begin{abstract}
 If a knot has the Alexander polynomial not equvalent to $1$, 
then it is linear $n$-colorable.
 By means of such a coloring, such a knot is given an upper bound for the minimal quandle order,
i.e., the minimal order of a quandle with which the knot is quandle colorable.
  For twist knots, we study the minimal quandle orders in detail.
\end{abstract}

\maketitle

\section{Introduction}\label{sect:introduction} 

 In 1982, Joyce \cite{Joyce} and Matveev \cite{Matveev} introduced the notion of quandles.  
A {\it quandle} is a non-empty set equipped with a binary operation satisfying three axioms that correspond to the three Reidemeister moves on knot diagrams. 
Alexander quandles form an important class of quandles.
See Section~\ref{sec:Fundamental properties} for details of quandles.

The history of colorings of knot diagrams goes back to Fox colorings \cite{F}, that correspond to homomorphisms of knot groups to the dihedral groups. 
It is well-known that if the determinant of a given knot is divisible by an integer $n$ with $n\geq 3$, then any diagram of the knot is
Fox $n$-colorable \cite{F}. 
Such colorings are generalized to quandle colorings (see Section \ref{sec:Fundamental properties}), 
and some conditions on the Alexander polynomial of a given knot so that its diagrams have non-trivial colorings with a given Alexander quandle are shown in \cite{Bae, Inoue}. 

In this paper, we study a generalization of Fox-colorings which corresponds to a class of Alexander quandle colorings. 

 Let $L$ be an oriented classical link in the $3$-sphere $S^3$, 
$D$ a link diagram of $L$ on a $2$-sphere $S^2$ in $S^3$,
and $c$ the number of crossings of $D$.
 Cutting $D$ at the $c$ undercrossing points,
we obtain $c$ arcs, which we call {\it strands} of $D$.
 Let ${\mathcal S}(D)$ be the set of strands of $D$.

\begin{figure}[htbp]
\begin{center}
\includegraphics[width=20mm]{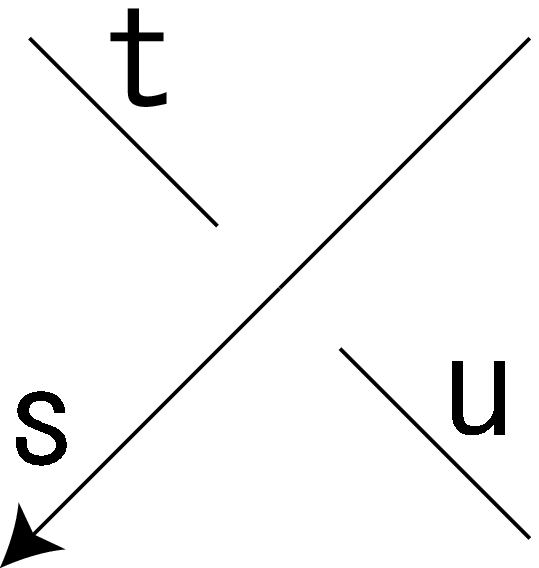}
\end{center}
\caption{}
\label{fig:CrossingCondition}
\end{figure}

 Let $n$ be an integer larger than $1$,
and $k$ and $\ell$ be positive integers
with ${\rm GCD}(n,k)=1$ and ${\rm GCD}(n,\ell)=1$, 
i.e., $n$ and $k$ are coprime and so are $n$ and $\ell$. 
 Let ${\mathbb Z}_n$ be the cyclic group of order $n$.
 A map $\varphi : {\mathcal S}(D) \rightarrow {\mathbb Z}_n$ is called 
a {\it linear $n$-coloring}, or more precisely,
a {\it $({\mathbb Z}_n, {}_{\ell}*_k)$-coloring}
if it satisfies the following {\it crossing condition}
at each crossing of $D$.
 Let $x$ be a crossing point of $D$, 
$s$ the strand which goes over $x$,
$t$ the strand which lies on the right hand of $s$,
and $u$ the strand on the left hand of $s$.
 See Figure \ref{fig:CrossingCondition}.
 Then the crossing condition at $x$ is given by the equation 
$\ell\varphi(t)+k\varphi(u)=(\ell+k)\varphi(s)$ in ${\mathbb Z}_n$.
 This notion coincides with the usual Fox $n$-coloring \cite{F}
when $k=1$ and $\ell=1$.
 For general $k$ and $\ell$,
this is no more than a coloring with the Alexander quandle $(\mathbb Z _n, {}_{\ell}*_k)$, 
see Section~\ref{sec:Fundamental properties}.

 A linear $n$-coloring is called {\it trivial}
if it is a constant map.
 If the diagram $D$ admits a non-trivial linear $n$-coloring,
the link $L$ is called {\it linear $n$-colorable} or {\it $(\mathbb Z _n, {}_{\ell}*_k)$-colorable}.
 The number of all $({\mathbb Z}_n, {}_{\ell}*_k)$-colorings
is called the {\it $({\mathbb Z}_n, {}_{\ell}*_k)$-coloring number}.
The $({\mathbb Z}_n, {}_{\ell}*_k)$-colorability and $({\mathbb Z}_n, {}_{\ell}*_k)$-coloring number 
are invariants of an oriented link.

 The $(3,5)$-torus knot ($10_{124}$ in Rolfsen's table \cite{Ro})
has the determinant $1$,
and does not have a non-trivial Fox $n$-coloring, refer to \cite{knotatlas}.
 However, it has a non-trivial $({\mathbb Z}_{31}, {}_3*_1)$-coloring
as in Figure \ref{fig:10_124},
which is also a $({\mathbb Z}_{31}, {}_1*_{21})$-coloring.

\begin{figure}[htbp]
\begin{center}
\includegraphics[width=45mm]{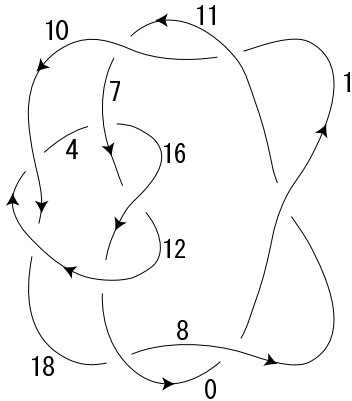}
\end{center}
\caption{the $(3,5)$-torus knot $10_{124}$}
\label{fig:10_124}
\end{figure}

 In this paper, $\Delta_L (t)$ denotes the Alexander polynomial of a knot $L$, 
where  $t$ is the variable.
 We can assume that the constant term of $\Delta_L (t)$ is of the lowest degree
by multiplying it by $t^i$ for some integer $i$. 
 Let $a_i$ be the coefficient of $t^i$ in $\Delta_L (t)$.
 That is, $\Delta_L (t) = a_d t^d + a_{d-1} t^{d-1} + \cdots + a_1 t + a_0$ with $a_0 \ne 0$.


 In the case where $n$ is a prime number, $k=1$ and $\ell=1$, 
(1) in the next proposition is well-known, refer to \cite{L}.
 Note that $n$ may be non-prime in (1).
 When $n$ is prime,
(1) is a part of Corollary 2 in \cite{Inoue} by A. Inoue.

\begin{proposition}\label{lem:FundamentalProperty}
\begin{enumerate}
\item[(1)] 
 Let $n$ be an integer larger than $1$,
and $k$, $\ell$ be positive integers coprime with $n$.
 A knot $L$ is $({\mathbb Z}_n, {}_{\ell} *_k)$-colorable 
if and only if there exists an odd prime natural number $p$ 
such that $p|n$ and  $\Delta_L(-\overline{\ell}k) \equiv 0$ (mod $p$), 
where $\overline{\ell}$ is the element of ${\mathbb Z}_n$ 
with $\overline{\ell}\ell=1$. 
\item[(2)]
 Let $n$ be an odd prime natural number.
 Suppose that a knot $L$ is $({\mathbb Z}_n, {}_1*_k)$-colorable.
 Then it is also $({\mathbb Z}_n, {}_1*_{\overline{k}})$-colorable,
where $\overline{k}$ is the element of ${\mathbb Z}_n$ with $\overline{k}k=1$.
 Moreover, $L$ is also $({\mathbb Z}_n, {}_k*_1)$-colorable. 
\item[(3)]
If a knot $L$ has the Alexander polynomial $\Delta_L (t) \ne 1$,
then it is $({\mathbb Z}_n, {}_1 *_k)$-colorable
for some odd integer $n$ larger than $2$ 
and some positive integer $k$ with ${\rm GCD}(n,k) \ne 1$.
In fact, this is the case for $k$ and $n=|\Delta_L (-k)|$,
where $k$ is a positive integer with (i) ${\rm GCD}(k, a_0)=1$
and (ii) $k \ge (\displaystyle\max_{1\le i \le d-1} |a_i /a_d|) + 1$.
\item[(4)]
For an arbitrarily given knot $L$ and an arbitrarily given integer $n$ larger than $2$,
there is a finite algorithm 
to determine if a given knot $L$ is linear $n$-colorable or not.
\item[(5)]
For an arbitrarily given knot $L$ with the Alexander polynomial $\Delta_L (t) \ne 1$, 
there is a finite algorithm
to determine the minimal number $n$
for which $L$ is linear $n$-colorable.\end{enumerate}
\end{proposition}

\begin{remark}\label{rem:colorable}
(1) in the above proposition implies the three facts below.
Let $n$, $k$ and $\ell$ be integers 
as in Proposition \ref{lem:FundamentalProperty} (1). 
\begin{enumerate}
\item[(a)] 
A knot $L$ is $({\mathbb Z}_n, {}_{\ell} *_k)$-colorable
if $\Delta_L (-\overline{\ell}k) \equiv 0$ (mod $n$).
The converse is also true if $n$ is an odd prime natural number.
\item[(b)]
If a knot $L$ is $({\mathbb Z}_n, {}_{\ell} *_k)$-colorable,
and $m$ is a multiple of $n$,
then $L$ is $({\mathbb Z}_m, {}_{\ell} *_k)$-colorable.
\item[(c)]
If a knot $L$ is $({\mathbb Z}_n, {}_{\ell} *_k)$-colorable,
then $L$ is $({\mathbb Z}_p, {}_{\ell} *_k)$-colorable
for some odd prime natural number $p$ by which $n$ is divisible.
\end{enumerate}
\end{remark}

\begin{remark}
 In Remark \ref{rem:colorable} (a),
the converse is not true if $n$ is not prime.
 Actually, 
when $L$ is the trefoil knot,
$L$ is $({\mathbb Z}_{15}, {}_1 *_1)$-colorable,
while $\Delta_L (t) = t^2 -t + 1$ 
and $\Delta_L (-1) = 3 \not\equiv 0$ (mod $15$).
\end{remark}

\begin{figure}[htbp]
\begin{center}
\includegraphics[width=50mm]{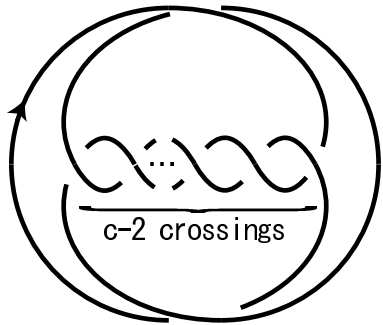}
\end{center}
\caption{twist knot with crossing number $c$}
\label{fig:twist}
\end{figure}

The {\it minimal quandle order} of a knot $L$, denoted by $q(L)$, 
is the minimal number among the orders of quandles 
each of which gives a non-trivial quandle coloring to a diagram of the knot, that is,
\[
q(L)={\rm min}\{ \mbox{the order of $Q$}~|~ Q\mbox{: a quandle such that $L$ is $Q$-colorable}\}.
\]

 In Section \ref{sect:minimal}, we study the minimal quandle orders for twist knots. We have the following theorem:

\begin{theorem}\label{theorem:twist}
Let $L$ be the twist knot with crossing number $c$ 
shown in Figure \ref{fig:twist}. 
\begin{enumerate}
\item[(1)] 
We have $q(L)=3$ if and only if  $c\equiv 0 $ $\pmod{3}$. 
In fact, $({\mathbb Z}_3,{}_1*_1)$ gives the minimal quandle order.  
\item[(2)] 
We have $q(L)=4$ if and only if  $c=12q+r$ for $r \in \{4,7,8,11\}$ and some non-negative integer $q$. 
In fact, the tetrahedron quandle gives the minimal quandle order. 
(See the beginnning of section 3 for the definition of the tetrahedron quandle.)
\item[(3)] 
We have $q(L)=5$ if and only if  $c=10(6q+r)+s$ for some non-negative integers $q,r,s$ 
with $(r,s) \in \{1,3\}\times \{4,7\}$ or $\{2,4\} \times \{6,9\}$. 
Especially when $s\in \{4,9\}$ (resp. $s\in \{6,7\}$), 
$({\mathbb Z}_5,{}_1*_1)$ (resp. $({\mathbb Z}_5,{}_1*_2)$) 
gives the minimal quandle order.
\item[(4)] 
We have $q(L)=7$ if and only if  $c=14(30q+r)+s$ for some non-negative integers $q,r,s$ 
with $ (r,s) \in \{ 0,4,10,12,22,24 \} \times \{5\}$, 
\newline
$\{ 1,3,13,15,21,25 \}\times \{11\}$, 
$\{  4,8,14,16,26,28 \}\times \{6\}$, 
\newline
$\{  5,7,13,17,23,25 \}\times \{0,3\}$ 
or $\{  5,7,17,19,25,29 \} \times \{12\}$.
Especially when $s\in \{5,12\}$ (resp. $\{0,3\}$ or $\{6,11\}$), 
$({\mathbb Z}_7,{}_1*_1)$ (resp. $({\mathbb Z}_7,{}_1*_2)$ 
or $({\mathbb Z}_7,{}_1*_3)$) gives the minimal quandle order.
\item[(5)] 
We have $q(L)\geq 8$, otherwise. 
\end{enumerate}
\end{theorem}

\section{Fundamental properties}\label{sec:Fundamental properties}

A {\it quandle} is defined to be a set $Q$ with a binary operation $* : Q \times  Q \to Q$ 
satisfying the following properties: 
(1) For any $a\in Q$, $a*a=a$. 
(2) For any pair of $a, b \in Q$, there exists a unique $c\in Q$ such that $c*b=a$. 
(3) For any triple of $a,b,c \in Q$, $(a*b)*c=(a*c)*(b*c)$.
 Let $Q_1, Q_2$ be quandles.
 A map $f : Q_1 \rightarrow Q_2$ is said to be a {\it homomorphism}
if $f(a*b)=f(a)*'f(b)$ holds for all $a,b \in Q_1$,
where $*$ and $*'$ are quandle operations in $Q_1$ and $Q_2$ respectively.
 If such a map $f$ is bijective,
then it is called an {\it isomorphism},
and we say that $Q_1$ and $Q_2$ are {\it isomorphic}. 

Alexander quandles form an important class of quandles.
Let $\Lambda $ be the Laurent polynomial ring $\mathbb Z[t, t^{-1}]$. 
Then, any $\Lambda $-module $M$ is a quandle with the operation $*: M\times M \to M$ defined by $a*b = ta+(1-t)b$, 
that is called an {\it Alexander quandle}.

Let $Q$ be a finite quandle. Let $D$ be a diagram of a given oriented classical link $L$, 
and let ${\mathcal S}(D)$ be the set of strands.  
A {\it (quandle) coloring} is a map $\psi:{\mathcal S} (D) \to Q$ such that at every crossing, 
the relation $\psi (t) * \psi(s) = \psi(u)$ holds, 
where $t$, $s$ and $u$ are the strands around the crossing located as shown in Figure~\ref{fig:CrossingCondition}. 
A quandle coloring is called {\it trivial} if it is a constant map. 
If $D$ admits a non-trivial quandle coloring, the link $L$ is called {\it $Q$-colorable} or {\it quandle colorable by $Q$}.   
It is well-known that $Q$-colorability is an invariant of a link.

 The crossing condition 
$\ell\varphi(t)+k\varphi(u)=(\ell+k)\varphi(s)$ in ${\mathbb Z}_n$
in Section \ref{sect:introduction}
can be deformed to
$\varphi(u) = ((\ell+k)\varphi(s)-\ell\varphi(t))\overline{k}$, 
where $\overline{k}$ means the element in $\mathbb Z_n$ with $\overline{k} k=1$. 
 If we define the operation ${}_{\ell}*_k$ in ${\mathbb Z}_n$
by $a\ {}_{\ell}*_k b = ((\ell+k)b-\ell a)\overline{k}$
for any $a, b \in {\mathbb Z}_n$,
then this operation gives a quandle operation in ${\mathbb Z}_n$. 
This implies that linear colorings are a kind of quandle coloring.  
In fact, 
the quandle $({\mathbb Z}_n,  {}_{\ell}*_k)$ is isomorphic 
to the Alexander quandle $\mathbb Z_n[t,t^{-1}]/(t+\ell \overline{k})$. 
(Since $a {}_{\ell}*_k b = (-\ell\bar{k})a + (1 - (-\ell\bar{k}))b$, 
the map $f : ({\mathbb Z}_n,  {}_{\ell}*_k) \to \mathbb Z_n[t,t^{-1}]/(t+\ell \overline{k})$ defined by $f(a)=a$ 
gives a quandle isomorphism.)
Hence we have the following lemma. 

\begin{lemma}\label{lem:AffineCyclicQuandle}
Let $\ell, \ell',  k, k'$ be positive integers 
each of which is coprime with $n$.
If $\ell \overline{k} = \ell' \overline{k'}$ holds, 
then the two quandles $({\mathbb Z}_n,  {}_{\ell}*_k)$ 
and $({\mathbb Z}_n,  {}_{\ell'}*_{k'})$ are isomorphic. 
This implies that the notion of $({\mathbb Z}_n, {}_\ell *_k)$-coloring 
is equivalent 
to that of $({\mathbb Z}_n, {}_{\ell\overline{k}}*_1)$-coloring 
(or $({\mathbb Z}_n, {}_1*{}_{\overline{\ell}k})$-coloring).
\end{lemma}

\begin{remark}\label{remark:Nelson}
Nelson \cite{Nelson} studied a class of Alexander quandles called {\it linear}, 
which corresponds to our quandles. 
By his result, 
we see that for positive integers $\ell$ and $k$ 
each of which is coprime with $n$,  
the quandle $({\mathbb Z}_n,  {}_{\ell}*_k)$ is trivial 
if and only if GCD$(n ,\ell + k )=n$,
and that 
the quandle $({\mathbb Z}_n,  {}_{\ell}*_k)$ is indecomposable (or connected)
if and only if GCD$(n ,\ell + k )=1$. 
In fact, 
when $0 < d =$GCD$(n ,\ell + k ) < n$, 
the quandle $({\mathbb Z}_n,  {}_{\ell}*_k)$ 
is decomposed into $d$ orbits
$C_i = \{ i, d+i, \cdots, ((n/d)-1)d+i \}$, 
$i =0,1,\cdots, d-1$,
since $i {}_{\ell}*_k j = \displaystyle\frac{\ell+k}{d}\bar{k} d(j-i) +i$.
See \cite{NW} for details of orbit decompositions.
The orbit $C_i$ forms a quandle under the operation ${}_{\ell}*_k$
isomorphic to $({\mathbb Z}_{n/d}, {}_{\ell}*_k)$
by the map $md+i \mapsto m$.
Since an image of a coloring of a knot is contained in an orbit
we can obtain a non-trivial linear $m$-coloring 
with an indecomposable quandle $({\mathbb Z}_m, {}_{\ell}*_k)$
from any non-trivial linear $n$-coloring
by repeating the operations of taking an orbit.
\end{remark}

\begin{figure}[htbp]
\begin{center}
\includegraphics[width=20mm]{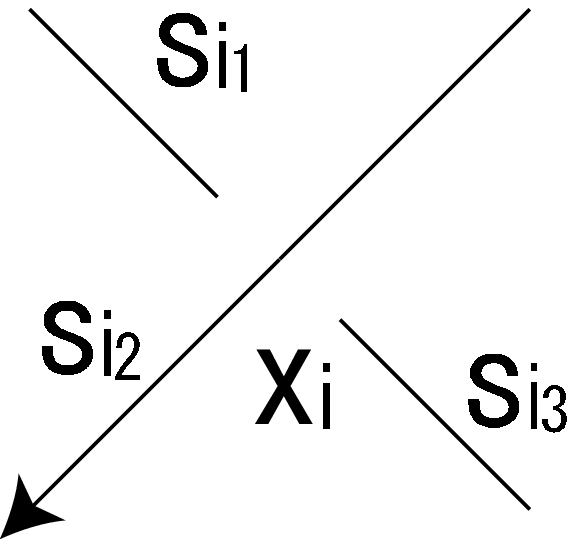}
\end{center}
\caption{}
\label{fig:CrossingCondition2}
\end{figure}

We prove Proposition \ref{lem:FundamentalProperty}. 
Note that for the Alexander polynomial of a knot 
with the form given in Section \ref{sect:introduction}, 
the following properties are well-known:  
(i) For any $i \in \{ 0, 1, 2, \cdots, d \}$, $a_{d-i} = a_i$  
and (ii) $\Delta_L(1)= a_d + a_{d-1} + \cdots + a_1 + a_0 = 1$. 
Additionally, by the properties (i) and (ii), 
we see that $d$ is an even number and $a_{d/2}$ is an odd number. 

We need to prove Remark \ref{rem:colorable} (a), (b) and (c) 
before proving Proposition \ref{lem:FundamentalProperty}(1).

\begin{proof}[Proof of Remark \ref{rem:colorable} (a)]
 By Lemma \ref{lem:AffineCyclicQuandle},
it is enough to consider the case of $\ell=1$.
 The proof is almost the same as that for the usual $n$-coloring with $k=1$ and $\ell=1$, refer to \cite{L}.

 Let $D$ be a diagram of the knot $L$.
 We number the crossings of $D$,
and call them $x_1, x_2, \cdots, x_c$.
 Let $s_1,s_2,\cdots,s_c$ be the strands of $D$.
 Then the crossing condition at $x_i$ is of the form
$\varphi(s_{i_1})+k\varphi(s_{i_3})=(k+1)\varphi(s_{i_2})$,
where $s_{i_2}$ is the strand going over the crossing $x_i$,
$s_{i_1}$ the strand on the right hand of $s_{i_2}$
and $s_{i_3}$ the strand on the left hand of $s_{i_2}$.
 See Figure \ref{fig:CrossingCondition2}.
 Let $M$ be  the coefficient matrix of the system of equalities of crossing condition for all the crossings.
 The $i_1$th element is $1$, 
the $i_2$th element is $-k-1$ 
and the $i_3$th element is $k$ 
in the $i$th row of $M$ when $i_1 \ne i_2$ and $i_2 \ne i_3$.
 If $i_1 = i_2$, then the $i_1$th element is $1+k$.
 The $i_3$th element is $k-(k+1)=-1$ if $i_2 = i_3$.
 Note that $M$ is a square matrix of order $c$.

 By Inoue \cite{Inoue}, it is proved that $M$ coincides 
with the Alexander matrix $A$ with $t$ replaced by $-k$,
where $A$ is calculated from the Wirtinger presentation of the fundamental group $G(L)=\pi_1(S^3-L)$
by Fox's derivative. See \cite{CF} for details of Fox's derivative.
 The Alexander polynomial $\Delta_L (t)$ is the greatest common divisor
of all the minor determinant of order $c-1$ of the Alexander matrix $A$.

 Hence, 
if $\Delta_L(-k)$ is a multiple of an integer $n$ larger than or equal to $2$,
then all the minor determinant of order $c-1$ of $M$ is equivalent to $0$ modulo $n$,
and hence the system of equations $M {\bf x} = {\bf 0}$ 
with the vector ${\bf x}$ of variables
has at least $n^2$ solutions.
 This implies that the coloring number is larger than or equal to $n^2$,
which shows that $L$ admits a non-trivial $({\mathbb Z}_n, {}_1*_k)$-coloring.

 Suppose that $n$ is an odd prime natural number,
and $L$ is $({\mathbb Z}_n, {}_1*_k)$-colorable.
 Then ${\mathbb Z}_n$ is a field.
 Since the system $M {\bf x} = {\bf 0}$ has more than $n$ solutions,
the rank of the solution space is $2$ or larger.
 Hence the rank of $M$ is smaller than or equal to $c-2$,
and all the minor determinant of order $c-1$ of $M$ 
is equivalent to $0$ modulo $n$. 
 Thus we have $\Delta_L(-k) \equiv 0$ (mod $n$).
\end{proof}

\begin{proof}[Proof of Remark \ref{rem:colorable} (b)]
A diagram $D$ of $L$ has 
a non-trivial $(\mathbb Z_{n}, {}_{\ell} * _k)$-coloring, say $\phi$. 
Define $\varphi : {\mathcal S}(D) \to \mathbb Z_m$ 
by $\varphi (s) = (m/n) \times \phi (s)'$, 
where $\phi (s)'$ means a representative element of $\phi (s)\in \mathbb Z_n$. 
We easily see that the map $\varphi $ is a non-trivial $(\mathbb Z_{m}, {}_{\ell} * _k)$-coloring of $D$.  
Thus, $L$ is $(\mathbb Z_{m}, {}_{\ell} * _k)$-colorable.
\end{proof}

\begin{proof}[Proof of Remark \ref{rem:colorable} (c)]
 The proof proceeds by induction on the number of prime factors of $n$.
 Let $D$ be a diagram of a knot,
and $\varphi : {\mathcal S}(D) \rightarrow {\mathbb Z}_n$ a non-trivial $({\mathbb Z}_n, {}_{\ell}*_k)$-coloring of $D$.
 We take an arbitrary prime divisor of $n$, say $p$.
 Let $\pi : {\mathbb Z}_n \ni z \mapsto {\rm Mod}[z,p] \in {\mathbb Z}_p$ be the natural projection map
with ${\rm Mod}[z,p] \equiv z$ (mod $p$).
 If the composition $\pi \circ \varphi$ is not a constant map,
then it gives a non-trivial $({\mathbb Z}_p, {}_{\ell}*_k)$-coloring.
 In this case, $p \ne 2$ 
since any linear $2$-coloroing is trivial for any knot.

 We consider the case where $\pi \circ \varphi$ is a constant map.
 Set $\pi \circ \varphi (s) = c$ for all strands $s \in {\mathcal S}(D)$,
where $c$ is a constant in ${\mathbb Z}_p$.
 Then $\varphi(s)-c$ is a multiple of $p$ for every $s \in {\mathcal S}(D)$,
and $\psi(s) = (\varphi(s)-c)/p$ gives a non-trivial coloring of $D$ by $({\mathbb Z}_{n/p}, {}_{\ell}*_k)$.
 Hence there is an odd prime natural number $p'$
such that $p'$ is a divisor of $n/p$ 
and that $D$ is linear $p'$-colorable with $({\mathbb Z}_{p'}, {}_{\ell}*_k)$
by the assumption of induction.
 This completes the proof.
\end{proof}

\begin{proof}[Proof of Proposition \ref{lem:FundamentalProperty}(1)]
If $L$ is $({\mathbb Z}_n, {}_{\ell}*_k)$-colorable,
then $L$ is $({\mathbb Z}_p, {}_{\ell}*_k)$-colorable
for some odd prime divisor $p$ of $n$ by Remark \ref{rem:colorable} (c).
Hence $\Delta_L (-k) \equiv 0$ (mod $p$) by Remark \ref{rem:colorable} (a).

If $\Delta_L (-k) \equiv 0$ (mod $p$) for some odd prime divisor $p$ of $n$,
then Remark \ref{rem:colorable} (a) assures
that $L$ is $(\mathbb Z_{p}, {}_{\ell} * _k)$-colorable. 
Hence $L$ is $(\mathbb Z_{n}, {}_{\ell} * _k)$-colorable 
by Remark \ref{rem:colorable} (b). 
\end{proof}

\begin{proof}[Proof of Proposition \ref{lem:FundamentalProperty}(2)]
 Let $\Delta_L (t) = a_d t^d + a_{d-1} t^{d-1} + \cdots + a_1 t + a_0$ 
be the Alexander polynomial of $L$.
 By the assumption that $L$ is $({\mathbb Z}_n, {}_1*_k)$-colorable and 
by Remark \ref{rem:colorable} (a),
we have 
$\Delta_L (-k) \equiv 0$ (mod $n$).

 On the other hand,
$\Delta_L (-\overline{k})
= a_d (-\overline{k})^d + a_{d-1}(-\overline{k})^{d-1} 
+ \cdots + a_1 (-\overline{k}) + a_0
= (-\overline{k})^d (a_d + a_{d-1} (-k) 
+ \cdots + a_1 (-k)^{d-1} + a_0 (-k)^d)$.
Since $a_{d-i} = a_i$ for any $i \in \{ 0, 1, 2, \cdots, d \}$, we have 
$\Delta_L (-\overline{k})=(-\overline{k})^d (a_0 + a_{1} (-k) 
+ \cdots + a_{d-1} (-k)^{d-1} + a_d (-k)^d)=(-\overline{k})^d \Delta_L (-k).$
This implies 
$\Delta_L (-\overline{k}) \equiv 0 \pmod{n}$.
Hence $L$ is $({\mathbb Z}_n, {}_1*_{\overline{k}})$-colorable 
by Remark \ref{rem:colorable} (a),
and $({\mathbb Z}_n, {}_k*_1)$-colorable by Lemma \ref{lem:AffineCyclicQuandle}.
\end{proof}

\begin{proof}[Proof of Proposition \ref{lem:FundamentalProperty}(3)]
 Let $L$ be a knot, 
and $\Delta_L (t)=a_d t^d + a_{d-1} t^{d-1} + \cdots + a_1 t + a_0$ 
the Alexander polynomial of $L$.
 Let $k$ be a positive integer 
with (i) ${\rm GCD}(k, a_0)=1$ 
and (ii) $k \ge (\displaystyle\max_{1\le i \le d-1} |a_i /a_d|) + 1$.
 Note that $k \ge 2$ because $a_{d/2}$ is odd and non-zero.

 We show that $|\Delta_L (-k)|$ is an odd integer larger than $k+1$, which is larger than $2$,
and that $|\Delta_L (-k)|$ is coprime with $k$.

 Since $\Delta_L (-k) \equiv a_0$ (mod $k$), 
we have ${\rm  GCD}(\Delta_L (-k), k)={\rm GCD}(a_0,k)=1$
by Euclidean method of mutual division and the condition (i).

 Next we show that  $\Delta_L (-k)$ is odd.
 When $k$ is odd,
$a_{d-i} (-k)^{d-i} + a_i (-k)^i 
\equiv a_{d-i} + a_i 
= 2 a_i
\equiv 0$ (mod $2$)
for any integer $i$ with $0 \le i < d/2$.
 Hence we have 
$\Delta_L (-k)
\equiv a_{d/2} (-k)^{d/2}
\equiv a_{d/2}$ (mod $2$).
 This and $1 = \Delta_L (1) \equiv a_{d/2}$ (mod $2$)
imply that $\Delta_L (-k)$ is odd.
 When $k$ is even,
$\Delta_L (-k)
\equiv a_0
\equiv 1$ (mod $2$).
 Note that $a_0$ is odd
since GCD$(a_0,k)=1$.

 We also show that $|\Delta_L (-k)| > k+1$.
By the condition (ii), we have 
\newline
$|a_d| (-k)^d 
=|a_d|k^d
\ge |a_d|(|a_{d-1}/a_d|+1)k^{d-1}
=|a_{d-1}|k^{d-1} + |a_d|k^{d-1}$
\newline
$\ \ \ge |a_{d-1}|k^{d-1} + |a_d|(|a_{d-2}/a_d|+1)k^{d-2}
=|a_{d-1}|k^{d-1} + |a_{d-2}|k^{d-2} + |a_d|k^{d-2}$
\newline
$\ \ \ge \cdots 
=|a_{d-1}|k^{d-1} + |a_{d-2}|k^{d-2} + \cdots + |a_1|k + |a_d|k\ \cdots$ (iii).
\newline
 Recall that $a_d = a_0$ for the Alexander polynomial.
 We consider the case where $a_d$ and $a_0$ are negative. 
 (A similar argument will do for the other case where $a_d$ and $a_0$ are positive.)
 Then, we have 
\newline
$\Delta_L (-k) 
=a_d (-k)^d + a_{d-1} (-k)^{d-1} + \cdots + a_1 (-k) + a_0$
\newline
$\ \ \le \{-|a_{d-1}|k^{d-1}-|a_{d-2}|k^{d-2}-\cdots-|a_1|k-|a_d|k \}$ 
\newline
$\ \ \ \ \ \ + a_{d-1}(-k)^{d-1} + a_{d-2}(-k)^{d-2} + \cdots + a_1(-k) + a_0$
\newline
$\ \ =(-|a_{d-1}|-a_{d-1})k^{d-1} + (-|a_{d-2}|+a_{d-2})k^{d-2}+\cdots+(-|a_1|-a_1)k-|a_d|k+a_0$
\newline
$\ \ \le -|a_d|k+a_0
= a_d k + a_0 = a_d(k+1) \le -(k+1)\ \cdots$(iv)
\newline
 We show that the equality $\Delta_L (-k) = -(k+1)$ does not hold.
 If equality holds on (iii), 
then $k=|a_i/a_d|+1$ holds for all $i$ with $1 \le i \le d-1$,
and hence there is an positive integer $m$
with $m |a_d| = |a_{d-1}| = |a_{d-2}| = \cdots = |a_2| = |a_1| = m |a_0|$.
 If the equality holds at the second inequality of (iv),
then $a_i = (-1)^i |a_i|$ holds for any $i$ with $1 \le i \le d-1$.
 If the equality holds at the third inequality of (iv),
then $a_0 = a_d = -1$.
 Thus the condition $\Delta_L (-k) = -(k+1)$ leads to
$a_d = -1,\ a_{d-1}=-m,\ a_{d-2}=+m, \cdots,\ a_2 = +m,\ a_1 = -m,\ a_0 = -1$,
which contradict $\Delta_L (1)=+1$.

 Put $n=|\Delta_L (-k)|$.
 Then, we have proved that $n$ is an odd integer 
with ${\rm GCD}(n,k)=1$ and $n > k+1 > 2$. 
 Hence $L$ is $({\mathbb Z}_{|\Delta_L (-k)|}, {}_1*_k)$-colorable  
by Remark \ref{rem:colorable} (a).
\end{proof}

\begin{proof}[Proof of Proposition \ref{lem:FundamentalProperty}(4)]
 If a knot $L$ is linear $n$-colorable,
then it is $({\mathbb Z}_n, {}_1*_k)$-colorable
for some integer $k$ with $1 \le k \le n-2$
by Lemma \ref{lem:AffineCyclicQuandle}
and Remark \ref{remark:Nelson},
and hence by Proposition \ref{lem:FundamentalProperty}(1), 
there exists an odd prime divisor $p$ of $n$
such that $\Delta_L (-k) \equiv 0$ (mod $p$).
 If there is a multiple of $p$ in $\{\Delta_L (-k)~|~ k= 1,2,\ldots,p-1 \}$ for some prime divisor $p$ of $n$, 
then $L$ is linear $n$-colorable.
 If not, then $L$ is not linear $n$-colorable.
\end{proof}

\begin{remark}
 Practically, we do not need to calculate all $\Delta_L (-k)$.
 For example, if $\Delta_L (-2) \not\equiv 0$ (mod $5$),
then $\Delta_L (-3) \not\equiv 0$ (mod $5$)
by  Proposition \ref{lem:FundamentalProperty} (2) and $2\cdot 3 \equiv 1$ (mod $5$).
 Note that $({\mathbb Z}_5, {}_1*_2)$ and $({\mathbb Z}_5, {}_1*_3)$
are not isomorphic as quandles.
\end{remark}

\begin{proof}[Proof of Proposition \ref{lem:FundamentalProperty}(5)]
 We perform the procedure 
described in the proof of Proposition \ref{lem:FundamentalProperty}(4)
for $n=3,4,5, \cdots$.
 This sequence of procedures terminates
when we find an integer $n$
such that the knot is linear $n$-colorable.
 Such an integer $n$ exists 
and is bounded explicitly by Proposition \ref{lem:FundamentalProperty}(3).
\end{proof}

\section{Minimal quandle order}\label{sect:minimal}

The Alexander quandle $\mathbb Z_2[t,t^{-1}]/(t^2+t+1)$ has four elements. 
It is isomorphic to the quandle consisting of $120$ degree rotations of a regular
tetrahedron, and hence the quandle is called the tetrahedron quandle and we denote it by $S_4$. 
A knot $L$ is colorable with the tetrahedron quandle
if and only if the Alexander polynomial $\Delta_L (t)$ of $L$ 
is equivalent to $0$ in ${\mathbb Z}_2[t,t^{-1}]/(t^2+t+1)$, 
see \cite{Inoue}.

It is known that indecomposable quandles of order at least $2$ and at most $5$
are 
$({\mathbb Z}_3, {}_1*_1)$,
$S_4$, 
$({\mathbb Z}_5, {}_1*_1)$,
$({\mathbb Z}_5, {}_1*_2)$,
$({\mathbb Z}_5, {}_1*_3)$, 
there are exactly two indecomposable quandles of order $6$, say $QS6$ and $QS6'$, 
and every indecomposable quandle of order $7$ is 
$({\mathbb Z}_7, {}_1*_1)$,
$({\mathbb Z}_7, {}_1*_2)$,
$({\mathbb Z}_7, {}_1*_3)$,
$({\mathbb Z}_7, {}_1*_4)$, or 
$({\mathbb Z}_7, {}_1*_5)$, 
refer to \cite{Nelson, V}. 
The following are the quandle operation matrices of $QS6$ and $QS6'$:
\begin{center}
$M_{QS6}=\left( \begin{array}{cccccc}
 1& 1& 5& 6& 3& 4\\
 2& 2& 6& 5& 4& 3\\
 5& 6& 3& 3& 1& 2\\
 6& 5& 4& 4& 2& 1\\
 3& 4& 1& 2& 5& 5\\
 4& 3& 2& 1& 6& 6
\end{array} \right)$, \ \ 
$M_{QS6'}=\left( \begin{array}{cccccc}
 1& 1& 6& 5& 3& 4\\
 2& 2& 5& 6& 4& 3\\
 5& 6& 3& 3& 2& 1\\
 6& 5& 4& 4& 1& 2\\
 4& 3& 1& 2& 5& 5\\
 3& 4& 2& 1& 6& 6
\end{array} \right)$,
\end{center}
where the $(i, j)$-element is $i*j$.
We have the next lemma:

\begin{lemma}\label{coloringbyQS6} 
If a knot is quandle colorable with an indecomposable quandle of order $6$, 
then it is $3$-colorable.
\end{lemma}

\begin{proof} 
The quandle $QS6$ is the extension of the quandle $({\mathbb Z}_3, {}_1*_1)$, 
that is, there is a surjective quandle homomorphism 
$f : QS6 \to ({\mathbb Z}_3, {}_1*_1)$ 
such that  for any element of $({\mathbb Z}_3, {}_1*_1)$, 
the cardinality of the inverse image by $f$ is constant. 
For example, define $f$ 
by $f(1)=0$, $f(2)=0$, $f(3)=1$, $f(4)=1$, $f(5)=2$ and $f(6)=2$. 
The inverse image of each element of $({\mathbb Z}_3, {}_1*_1)$ 
forms a trivial subquandle with two elements.
 
For any knot diagram $D$ and any non-trivial coloring $C$ of $D$ by $QS6$, 
put the composition $f\circ C$. 
It is a quandle coloring of $D$ by $({\mathbb Z}_3, {}_1*_1)$. 
We show that the coloring is non-trivial:
Assume that the coloring $f\circ C$ is trivial 
and denote by $a$ the element in ${\rm Im } (f\circ C)$. 
Then we have ${\rm Im }C \subset f^{-1}(a)$, 
which implies that the coloring $C$ is essentially the quandle coloring 
by the subquandle $f^{-1}(a)$ forms. 
However it leads to a contradiction since the subquandle is a trivial quandle, 
and hence $f\circ C$ is a non-trivial coloring. 
Therefore any $QS6$-colorable knot is $3$-colorable. 

The quandle $QS6'$ has also the same properties as $QS6$ has. 
Hence by the same argument, 
we see that any $QS6'$-colorable knot is $3$-colorable.
\end{proof}

Lemma~\ref{coloringbyQS6} implies that for a knot $L$ 
which is quandle colorable with an indecomposable quandle of order $6$, 
$q(L)=3$ holds.  



 We consider twist knots as shown in Figure \ref{fig:twist}.

\begin{lemma}\label{lem:twist1}
 Let $n$ be a prime natural number, 
$k$ a natural number with 
\newline
GCD$(n,k)=1$ and GCD$(n,k+1)\not =n$
and $c$ a natural number with $c \ge 3$.
 A twist knot with crossing number $c$
is $({\mathbb Z}_n, {}_1*_k)$-colorable
if and only if
either 
$c=2nm+2p$ 
for some non-negative integers $m$ and $p$ 
with $p \equiv \overline{k+1}^2(k^2+k+1)$ (mod $n$) and $ p \le n-1$, 
or $c=2nm+2p+1$ for some non-negative integers $m$ and $p$ 
with $p \equiv \overline{k+1}^2k$ (mod $n$) and $p \le n-1$,
where $\overline{k+1}$ is the element of ${\mathbb Z}_n$
with $\overline{k+1}(k+1)=1$.
 Even when $n$ is not prime, the $\lq\lq$if part" is true.
\end{lemma}

\begin{proof}
 Let $L$ be the twist knot with the crossing number $c$.
 When $c$ is even,
$c=2nm + 2p$ for some
non-negative integers $m$ and $p$ with $0 \le p \le n-1$. 
An easy calculation shows 
that the Alexander polynomial of $L$ is given by the formula
$\Delta_L (t) = -((c-2)/2)t^2+(c-1)t-(c-2)/2\ \cdots$ (i). 
 Substituting $2nm+2p$ for $c$,
and $-k$ for $t$  in the polynomial function (i), we have 
\newline
$\Delta_L (-k) =-((2nm+2p-2)/2)(-k)^2+(2nm+2p-1)(-k)-(2nm+2p-2)/2$
\newline
$\ \ \equiv -(p-1)k^2-(2p-1)k-(p-1)
=-(k+1)^2p+(k^2+k+1)$ (mod $n$).
\newline
 Hence $\Delta_L (-k) \equiv 0$ (mod $n$)
if and only if $p \equiv \overline{k+1}^2(k^2+k+1)$ (mod $n$).
 Then the desired conclusion follows 
from Remark \ref{rem:colorable} (a).

 We consider the case where $c$ is odd.
 Then $c=2nm+2p+1$
for some non-negative integers $m$ and $p$ with $0 \le p \le n-1$.
 The Alexander polynomial is given by
$\Delta_L (t) = ((c-1)/2)t^2-(c-2)t+(c-1)/2$.
 Then
\newline
$\Delta_L (-k) = ((2nm+2p+1-1)/2)(-k)^2-(2nm+2p+1-2)(-k)+(2nm+2p+1-1)/2$
\newline
$\ \ \equiv pk^2+(2p-1)k+p
=(k+1)^2p-k$ (mod $n$).
\newline
 Hence $\Delta_L (-k) \equiv 0$ (mod $n$)
if and only if $p \equiv \overline{k+1}^2k$ (mod $n$).
\end{proof}

\begin{lemma}\label{lem:twist2}
 Let $c$ be a natural number with $c \ge 3$.
 A twist knot with crossing number $c$
is colorable with the tetrahedron quandle
if and only if
$c=4m-1$ or $4m$ 
for some natural number $m$.
\end{lemma}

\begin{proof}
 When $c$ is even, 
there is a natural number $m$ and an integer $p$ with $p=0$ or $1$
such that $c=4m-2p$.
 Then
$\Delta_L (t) = -((c-2)/2)t^2 + (c-1)t - (c-2)/2
= -((4m-2p-2)/2)t^2 + (4m-2p-1)t -(4m-2p-2)/2
\equiv (p+1)t^2+t+(p+1)$ in ${\mathbb Z}_2[t,t^{-1}]$,
which is equivalent to $0$
in ${\mathbb Z}_2[t,t^{-1}]/(t^2+t+1)$
if and only if $p=0$.

 When $c$ is odd, 
$c=4m-2p+1$
for integers $m$, $p$ with $m > 0$ and $p=0$ or $1$.
 Then
$\Delta_L (t) = ((c-1)/2)t^2 - (c-2)t + (c-1)/2
= -((4m-2p+1-1)/2)t^2 + (4m-2p+1-2)t -(4m-2p+1-1)/2
\equiv pt^2+t+p$ in ${\mathbb Z}_2[t,t^{-1}]$,
which is equivalent to $0$
in ${\mathbb Z}_2[t,t^{-1}]/(t^2+t+1)$
if and only if $p=1$.
\end{proof}

\begin{proof}[Proof of Theorem \ref{theorem:twist}]
For a knot $L$, we have 
$q(L)= 3$ if and only if it is quandle colorable by $({\mathbb Z}_3,{}_1*_1)$,
$q(L)=4$ if and only if it is quandle colorable by $S_4$, but not by $({\mathbb Z}_3,{}_1*_1)$,
$q(L)=5$ if and only if it is quandle colorable by $({\mathbb Z}_5,{}_1*_1)$ or $({\mathbb Z}_5,{}_1*_2)$, 
but not by $({\mathbb Z}_3,{}_1*_1)$ nor $S_4$, 
$q(L)=7$ 
if and only if it is quandle colorable by $({\mathbb Z}_7,{}_1*_1)$, $({\mathbb Z}_7,{}_1*_2)$ or $({\mathbb Z}_7,{}_1*_3)$, 
but not by $({\mathbb Z}_3,{}_1*_1)$, $S_4$, $({\mathbb Z}_5,{}_1*_1)$ nor $({\mathbb Z}_5,{}_1*_2)$, 
and $q(L)\geq 8$, otherwise.
%
 Note that since a knot is 
$({\mathbb Z}_5, {}_1*_3)$-colorable (resp. $({\mathbb Z}_7, {}_1*_4)$-colorable or $({\mathbb Z}_7, {}_1*_5)$-colorable) 
if and only if it is 
$({\mathbb Z}_5, {}_1*_2)$-colorable (resp. $({\mathbb Z}_7, {}_1*_2)$-colorable or $({\mathbb Z}_7, {}_1*_3)$-colorable) 
by Proposition \ref{lem:FundamentalProperty} (2), 
we can omit to consider 
$({\mathbb Z}_5, {}_1*_3)$-colorings, $({\mathbb Z}_7, {}_1*_4)$-colorings and $({\mathbb Z}_7, {}_1*_5)$-colorings. 
 By Lemma \ref{coloringbyQS6}, we can also omit to consider QS6-colorings and QS6'-colorings.
 For twist knots, we consider each case in more detail.
 
 We consider only the case where the crossing number $c$ is even.
 (A similar calculation will do for the case where $c$ is odd, and we omit it.) 
 Then the twist knot with crossing number $c$ is 
$({\mathbb Z}_3,{}_1*_1)$-colorable
if and only if $c=6m_1$ for some positive integer $m_1$,
$({\mathbb Z}_5,{}_1*_1)$-colorable
if and only if $c=10m_2+4$ for some non-negative integer $m_2$,
$({\mathbb Z}_5,{}_1*_2)$-colorable
if and only if $c=10m_3+6$ for some non-negative integer $m_3$,
$({\mathbb Z}_7,{}_1*_1)$-colorable
if and only if $c=14m_4+12$ for some non-negative integer $m_4$,
$({\mathbb Z}_7,{}_1*_2)$-colorable
if and only if $c=14m_5$ for some positive integer $m_5$,
$({\mathbb Z}_7,{}_1*_3)$-colorable
if and only if $c=14m_6+6$ for some non-negative integer $m_6$
by Lemma \ref{lem:twist1}.
 The knot is colorable with the tetrahedron quandle
if and only if $c=4m$ for some positive integer $m$
by Lemma \ref{lem:twist2}.

 When $L$ is $S_4$-colorable
and simultaneously $({\mathbb Z}_3,{}_1*{}_1)$-colorable, 
$4m=c=6m_1$. This implies that $m=3q_1$ for some 
positive integer $q_1$.
Thus $S_4$ is a quandle which gives the minimal quandle order 
if and only if $m=3q+r$ for $r\in \{1,2\}$ and some non-negative integer $q$.
 
 When $L$ is $({\mathbb Z}_5,{}_1*{}_1)$-colorable
and simultaneously $({\mathbb Z}_3,{}_1*{}_1)$-colorable,
$10m_2 + 4 = c = 6m_1$.
 This implies that $5m_2 = 3m_1 -2 \equiv 1$ (mod $3$).
 Hence $m_2 = 3q_1+2$ for some non-negative integer $q_1$.
 When $L$ is $({\mathbb Z}_5,{}_1*{}_1)$-colorable
and simultaneously colorable with the tetrahedron quandle,
$10m_2 + 4 = c = 4m$.
 This implies that $5m_2 = 2m-2 \equiv 0$ (mod $2$).
 Hence $m_2 = 2q_2$ for some non-negative integer $q_2$.
 Thus $({\mathbb Z}_5,{}_1*_1)$ is a quandle which gives the minimal quandle order
if and only if $m_2 = 6q + r$ 
for $r \in \{ 1,3 \}$ and some non-negative integer $q$.

 When $L$ is $({\mathbb Z}_5,{}_1*{}_2)$-colorable
and simultaneously $({\mathbb Z}_3,{}_1*{}_1)$-colorable,
$10m_3 + 6 = c = 6m_1$.
 This implies that $5m_3 = 3m_1 -3 \equiv 0$ (mod $3$).
 Hence $m_3 = 3q_1$ for some non-negative integer $q_1$.
 When $L$ is $({\mathbb Z}_5,{}_1*{}_2)$-colorable
and simultaneously colorable with the tetrahedron quandle,
$10m_3 + 6 = c = 4m$.
 This implies that $5m_3 = 2m-3 \equiv 1$ (mod $2$).
 Hence $m_3 = 2q_2+1$ for some non-negative integer $q_2$.
 Thus $({\mathbb Z}_5,{}_1*_2)$ is a quandle which gives the minimal quandle order
if and only if $m_3 = 6q + r$ 
for $r \in \{ 2,4 \}$ and some non-negative integer $q$.

 When $L$ is $({\mathbb Z}_7,{}_1*{}_1)$-colorable
and simultaneously $({\mathbb Z}_3,{}_1*{}_1)$-colorable,
$14m_4 + 12 = c = 6m_1$.
 This implies that $7m_4 = 3m_1 -6 \equiv 0$ (mod $3$).
 Hence $m_4 = 3q_1$ for some non-negative integer $q_1$.
 When $L$ is $({\mathbb Z}_7,{}_1*{}_1)$-colorable
and simultaneously colorable with the tetrahedron quandle,
$14m_4 + 12 = c = 4m$.
 This implies that $7m_4 = 2m-6 \equiv 0$ (mod $2$).
 Hence $m_4 = 2q_2$ for some non-negative integer $q_2$.
 When $L$ is $({\mathbb Z}_7,{}_1*{}_1)$-colorable
and simultaneously $({\mathbb Z}_5,{}_1*{}_1)$-colorable,
$14m_4 + 12 = c = 10m_2+4$.
 This implies that $7m_4 = 5m_2 -4 \equiv 1$ (mod $5$).
 Hence $m_4 = 5q_3+3$ for some non-negative integer $q_3$.
 When $L$ is $({\mathbb Z}_7,{}_1*{}_1)$-colorable
and simultaneously $({\mathbb Z}_5,{}_1*{}_2)$-colorable,
$14m_4 + 12 = c = 10m_3+6$.
 This implies that $7m_4 = 5m_3 -3 \equiv 2$ (mod $5$).
 Hence $m_4 = 5q_4+1$ for some non-negative integer $q_4$.
 Thus $({\mathbb Z}_7,{}_1*_1)$ is a quandle which gives the minimal quandle order
if and only if $m_4 = 30q + r$ 
for $r \in \{  5,7,17,19,25,29 \}$ and some non-negative integer $q$.

 When $L$ is $({\mathbb Z}_7,{}_1*{}_2)$-colorable
and simultaneously $({\mathbb Z}_3,{}_1*{}_1)$-colorable,
$14m_5 = c = 6m_1$.
 This implies that $7m_5 = 3m_1 \equiv 0$ (mod $3$).
 Hence $m_5 = 3q_1$ for some 
positive integer $q_1$.
 When $L$ is $({\mathbb Z}_7,{}_1*{}_2)$-colorable
and simultaneously colorable with the tetrahedron quandle,
$14m_5 = c = 4m$.
 This implies that $7m_5 = 2m \equiv 0$ (mod $2$).
 Hence $m_5 = 2q_2$ for some 
positive integer $q_2$.
 When $L$ is $({\mathbb Z}_7,{}_1*{}_2)$-colorable
and simultaneously $({\mathbb Z}_5,{}_1*{}_1)$-colorable,
$14m_5 = c = 10m_2+4$.
 This implies that $7m_5 = 5m_2 +2 \equiv 2$ (mod $5$).
 Hence $m_5 = 5q_3+1$ for some non-negative integer $q_3$.
 When $L$ is $({\mathbb Z}_7,{}_1*{}_2)$-colorable
and simultaneously $({\mathbb Z}_5,{}_1*{}_2)$-colorable,
$14m_5 = c = 10m_3+6$.
 This implies that $7m_5 = 5m_3 +3 \equiv 3$ (mod $5$).
 Hence $m_5 = 5q_4+4$ for some non-negative integer $q_4$.
 Thus $({\mathbb Z}_7,{}_1*_2)$ is a quandle which gives the minimal quandle order
if and only if $m_5 = 30q + r$ 
for $r \in \{  5,7,13,17,23,25 \}$ and some non-negative integer $q$.

 When $L$ is $({\mathbb Z}_7,{}_1*{}_3)$-colorable
and simultaneously $({\mathbb Z}_3,{}_1*{}_1)$-colorable,
$14m_6 +6 = c = 6m_1$.
 This implies that $7m_6 = 3m_1 -3 \equiv 0$ (mod $3$).
 Hence $m_6 = 3q_1$ for some non-negative integer $q_1$.
 When $L$ is $({\mathbb Z}_7,{}_1*{}_3)$-colorable
and simultaneously colorable with the tetrahedron quandle,
$14m_6 +6 = c = 4m$.
 This implies that $7m_6 = 2m -3 \equiv 1$ (mod $2$).
 Hence $m_6 = 2q_2+1$ for some non-negative integer $q_2$.
 When $L$ is $({\mathbb Z}_7,{}_1*{}_3)$-colorable
and simultaneously $({\mathbb Z}_5,{}_1*{}_1)$-colorable,
$14m_6 +6 = c = 10m_2+4$.
 This implies that $7m_6 = 5m_2 -1 \equiv 4$ (mod $5$).
 Hence $m_6 = 5q_3+2$ for some non-negative integer $q_3$.
 When $L$ is $({\mathbb Z}_7,{}_1*{}_3)$-colorable
and simultaneously $({\mathbb Z}_5,{}_1*{}_2)$-colorable,
$14m_6 +6 = c = 10m_3+6$.
 This implies that $7m_6 = 5m_3 \equiv 0$ (mod $5$).
 Hence $m_6 = 5q_4$ for some non-negative integer $q_4$.
 Thus $({\mathbb Z}_7,{}_1*_3)$ is a quandle which gives the minimal quandle order
if and only if $m_6 = 30q + r$ 
for $r \in \{  4,8,14,16,26,28 \}$ and some non-negative integer $q$.

 Next we consider the case where $c$ is odd.
Then the twist knot with crossing number $c$ is 
$({\mathbb Z}_3,{}_1*_1)$-colorable
if and only if $c=6m_1+3$ for some non-negative integer $m_1$,
$({\mathbb Z}_5,{}_1*_1)$-colorable
if and only if $c=10m_2+9$ for some non-negative integer $m_2$,
$({\mathbb Z}_5,{}_1*_2)$-colorable
if and only if $c=10m_3+7$ for some non-negative integer $m_3$,
$({\mathbb Z}_7,{}_1*_1)$-colorable
if and only if $c=14m_4+5$ for some non-negative integer $m_4$,
$({\mathbb Z}_7,{}_1*_2)$-colorable
if and only if $c=14m_5+3$ for some positive integer $m_5$,
$({\mathbb Z}_7,{}_1*_3)$-colorable
if and only if $c=14m_6+11$ for some non-negative integer $m_6$
by Lemma \ref{lem:twist1}.
 The knot is colorable with the tetrahedron quandle
if and only if $c=4m-1$ for some positive integter $m$
by Lemma \ref{lem:twist2}.

 When $L$ is $S_4$-colorable
and simultaneously $({\mathbb Z}_3,{}_1*{}_1)$-colorable, 
$4m-1=c=6m_1+3$. This implies that $m=3q_1+1$ for some non-negative integer $q_1$.
Thus $S_4$ is a quandle which gives the minimal quandle order if and only if $m=3q+r$ for $r\in \{2,3\}$ and some non-negative integer $q$.

 When $L$ is $({\mathbb Z}_5,{}_1*{}_1)$-colorable
and simultaneously $({\mathbb Z}_3,{}_1*{}_1)$-colorable,
$10m_2 + 9 = c = 6m_1+3$.
 This implies that $5m_2 = 3m_1 -3 \equiv 0$ (mod $3$).
 Hence $m_2 = 3q_1$ for some non-negative integer $q_1$.
 When $L$ is $({\mathbb Z}_5,{}_1*{}_1)$-colorable
and simultaneously colorable with the tetrahedron quandle,
$10m_2 + 9 = c = 4m-1$.
 This implies that $5m_2 = 2m-5 \equiv 1$ (mod $2$).
 Hence $m_2 = 2q_2+1$ for some non-negative integer $q_2$.
 Thus $({\mathbb Z}_5,{}_1*_1)$ is a quandle which gives the minimal quandle order
if and only if $m_2 = 6q + r$ 
for $r \in \{ 2,4 \}$ and some non-negative integer $q$.

 When $L$ is $({\mathbb Z}_5,{}_1*{}_2)$-colorable
and simultaneously $({\mathbb Z}_3,{}_1*{}_1)$-colorable,
$10m_3 + 7 = c = 6m_1 +3$.
 This implies that $5m_3 = 3m_1 -2 \equiv 1$ (mod $3$).
 Hence $m_3 = 3q_1+2$ for some non-negative integer $q_1$.
 When $L$ is $({\mathbb Z}_5,{}_1*{}_2)$-colorable
and simultaneously colorable with the tetrahedron quandle,
$10m_3 + 7 = c = 4m-1$.
 This implies that $5m_3 = 2m-4 \equiv 0$ (mod $2$).
 Hence $m_3 = 2q_2$ for some non-negative integer $q_2$.
 Thus $({\mathbb Z}_5,{}_1*_2)$ is a quandle which gives the minimal quandle order
if and only if $m_3 = 6q + r$ 
for $r \in \{ 1,3 \}$ and some non-negative integer $q$.

 When $L$ is $({\mathbb Z}_7,{}_1*{}_1)$-colorable
and simultaneously $({\mathbb Z}_3,{}_1*{}_1)$-colorable,
$14m_4 + 5 = c = 6m_1+3$.
 This implies that $7m_4 = 3m_1 -1 \equiv 2$ (mod $3$).
 Hence $m_4 = 3q_1+2$ for some non-negative integer $q_1$.
 When $L$ is $({\mathbb Z}_7,{}_1*{}_1)$-colorable
and simultaneously colorable with the tetrahedron quandle,
$14m_4 + 5 = c = 4m-1$.
 This implies that $7m_4 = 2m-3 \equiv 1$ (mod $2$).
 Hence $m_4 = 2q_2+1$ for some non-negative integer $q_2$.
 When $L$ is $({\mathbb Z}_7,{}_1*{}_1)$-colorable
and simultaneously $({\mathbb Z}_5,{}_1*{}_1)$-colorable,
$14m_4 + 5 = c = 10m_2+9$.
 This implies that $7m_4 = 5m_2 +2 \equiv 2$ (mod $5$).
 Hence $m_4 = 5q_3+1$ for some non-negative integer $q_3$.
 When $L$ is $({\mathbb Z}_7,{}_1*{}_1)$-colorable
and simultaneously $({\mathbb Z}_5,{}_1*{}_2)$-colorable,
$14m_4 + 5 = c = 10m_3+7$.
 This implies that $7m_4 = 5m_3 +1 \equiv 1$ (mod $5$).
 Hence $m_4 = 5q_4+3$ for some non-negative integer $q_4$.
 Thus $({\mathbb Z}_7,{}_1*_1)$ is a quandle which gives the minimal quandle order
if and only if $m_4 = 30q + r$ 
for $r \in \{ 0,4,10,12,22,24 \}$ and some non-negative integer $q$.

 When $L$ is $({\mathbb Z}_7,{}_1*{}_2)$-colorable
and simultaneously $({\mathbb Z}_3,{}_1*{}_1)$-colorable,
$14m_5 + 3 = c = 6m_1+3$.
 This implies that $7m_5 = 3m_1  \equiv 0$ (mod $3$).
 Hence $m_4 = 3q_1$ for some non-negative integer $q_1$.
 When $L$ is $({\mathbb Z}_7,{}_1*{}_2)$-colorable
and simultaneously colorable with the tetrahedron quandle,
$14m_5 + 3 = c = 4m-1$.
 This implies that $7m_5 = 2m-2 \equiv 0$ (mod $2$).
 Hence $m_5 = 2q_2$ for some non-negative integer $q_2$.
 When $L$ is $({\mathbb Z}_7,{}_1*{}_2)$-colorable
and simultaneously $({\mathbb Z}_5,{}_1*{}_1)$-colorable,
$14m_5 + 3 = c = 10m_2+9$.
 This implies that $7m_5 = 5m_2 +3 \equiv 3$ (mod $5$).
 Hence $m_5 = 5q_3+4$ for some non-negative integer $q_3$.
 When $L$ is $({\mathbb Z}_7,{}_1*{}_2)$-colorable
and simultaneously $({\mathbb Z}_5,{}_1*{}_2)$-colorable,
$14m_5 + 3 = c = 10m_3+7$.
 This implies that $7m_5 = 5m_3 +2 \equiv 2$ (mod $5$).
 Hence $m_5 = 5q_4+1$ for some non-negative integer $q_4$.
 Thus $({\mathbb Z}_7,{}_1*_2)$ is a quandle which gives the minimal quandle order
if and only if $m_5 = 30q + r$ 
for $r \in \{ 5,7,13,17,23,25 \}$ and some non-negative integer $q$.

 When $L$ is $({\mathbb Z}_7,{}_1*{}_3)$-colorable
and simultaneously $({\mathbb Z}_3,{}_1*{}_1)$-colorable,
$14m_6 + 11 = c = 6m_1+3$.
 This implies that $7m_6 = 3m_1 -4  \equiv 2$ (mod $3$).
 Hence $m_6 = 3q_1 +2$ for some non-negative integer $q_1$.
 When $L$ is $({\mathbb Z}_7,{}_1*{}_3)$-colorable
and simultaneously colorable with the tetrahedron quandle,
$14m_6 + 11 = c = 4m-1$.
 This implies that $7m_6 = 2m-6 \equiv 0$ (mod $2$).
 Hence $m_6 = 2q_2$ for some non-negative integer $q_2$.
 When $L$ is $({\mathbb Z}_7,{}_1*{}_3)$-colorable
and simultaneously $({\mathbb Z}_5,{}_1*{}_1)$-colorable,
$14m_6 + 11 = c = 10m_2+9$.
 This implies that $7m_6 = 5m_2 -1 \equiv 4$ (mod $5$).
 Hence $m_6 = 5q_3+2$ for some non-negative integer $q_3$.
 When $L$ is $({\mathbb Z}_7,{}_1*{}_3)$-colorable
and simultaneously $({\mathbb Z}_5,{}_1*{}_2)$-colorable,
$14m_6 + 11 = c = 10m_3+7$.
 This implies that $7m_6 = 5m_3 -2 \equiv 3$ (mod $5$).
 Hence $m_6 = 5q_4+4$ for some non-negative integer $q_4$.
 Thus $({\mathbb Z}_7,{}_1*_3)$ is a quandle which gives the minimal quandle order
if and only if $m_6 = 30q + r$ 
for $r \in \{ 1,3,13,15,21,25 \}$ and some non-negative integer $q$.

The above argument yields the conclusion of Theorem \ref{theorem:twist}.
\end{proof}


\bibliographystyle{amsplain}

\medskip

\noindent
Chuichiro Hayashi: 
Department of Mathematical and Physical Sciences,
Faculty of Science, Japan Women's University,
2-8-1 Mejirodai, Bunkyo-ku, Tokyo, 112-8681, Japan.
hayashic@fc.jwu.ac.jp

\vspace{3mm}
\noindent
Miwa Hayashi:
Department of Mathematical and Physical Sciences,
Faculty of Science, Japan Women's University,
2-8-1 Mejirodai, Bunkyo-ku, Tokyo, 112-8681, Japan.
\newline
miwakura@fc.jwu.ac.jp

\vspace{3mm}
\noindent
Kanako Oshiro:
Department of Mathematical and Physical Sciences,
Faculty of Science, Japan Women's University,
2-8-1 Mejirodai, Bunkyo-ku, Tokyo, 112-8681, Japan.
\newline
ooshirok@fc.jwu.ac.jp

\end{document}